\documentclass[12pt,reqno]{amsart}
\usepackage{amsmath, amsthm, amssymb}

\usepackage{comment}
\usepackage{color}

\advance \topmargin by -\headheight
\advance \topmargin by -\headsep
     
\evensidemargin \oddsidemargin
\newtheorem{theorem}{Theorem}[section]
\newtheorem{lemma}[theorem]{Lemma}

\theoremstyle{definition}

\theoremstyle{remark}

\numberwithin{equation}{section}

\def\bfa{{\mathbf a}}
\def\bfb{{\mathbf b}}

 \def\bfe{{\mathbf e}}

\def\bfF{{\mathbf F}}

\def\bfj{{\mathbf j}}
\def\bfk{{\mathbf k}}

\def\bfr{{\mathbf r}}
\def\bfu{{\mathbf u}}
\def\bfv{{\mathbf v}}
\def\bfw{{\mathbf w}}
\def\bfx{{\mathbf x}}
\def\bfy{{\mathbf y}}
\def\bfz{{\mathbf z}}

\def\dtil{{\tilde{d}}}
\def\bfvtil{{\widetilde{\bfv}}}
\def\bfwtil{{\widetilde{\bfw}}}
\def\bfxtil{{\widetilde{\bfx}}}
\def\bfytil{{\widetilde{\bfy}}}
\def\bfxhat{{\widehat{\bfx}}}
\def\bfyhat{{\widehat{\bfy}}}

\def\calA{{\mathcal A}}  

\def\calB{{\mathcal B}}

\def\calD{{\mathcal D}}

\def\calF{{\mathcal F}}

\def\calL{{\mathcal L}}

\def\calU{{\mathcal U}}
\def\calV{{\mathcal V}}

\def\A{{\mathbb A}}
\def\C{{\mathbb C}}\def\P{{\mathbb P}}
\def\R{{\mathbb R}}
\def\Z{{\mathbb Z}}\def\Q{{\mathbb Q}}

\def\grM{{\mathfrak M}}

\def\grS{{\mathfrak S}}

\def\alp{{\alpha}} \def\bfalp{{\boldsymbol \alpha}}
\def\bet{{\beta}}  \def\bfbet{{\boldsymbol \beta}}
\def\gam{{\gamma}} 
\def\gamtil{{\tilde{\gam}}}
\def\Gam{{\Gamma}}
\def\del{{\delta}} 
\def\deltil{{\widetilde \delta}}

\def\tet{{\theta}}  
\def\vartet{{\vartheta}}
\def\kap{{\kappa}}
\def\lam{{\lambda}} \def\Lam{{\Lambda}}

\def\bfxi{{\boldsymbol \xi}}

\def\sig{{\sigma}}

\def\ome{{\omega}} 
\def\d{{\partial}}
\def\eps{\varepsilon}

\def\d{{\,{\rm d}}}

\def\meas{{\rm meas}}
\def\rank{{\rm rank}}
\def\Mtil{{\widetilde{M}}}
\def\na{{n_1}}
\def\nb{{n_2}}
\def\nor{{\rm nor}}

\newenvironment{blue}{\color{blue}}{}

\specialcomment{comblue}{\begin{blue}}{\end{blue}}
\excludecomment{comblue}

\begin{document}
\title[Bihomogeneous forms in many variables]{Bihomogeneous forms in many variables}
\author[Damaris Schindler]{Damaris Schindler}
\address{School of Mathematics, University of Bristol, University Walk, Clifton, Bristol BS8 1TW, United Kingdom}
\email{maxds@bristol.ac.uk}


\begin{abstract}
We count integer points on bihomogeneous varieties using the Hardy-Littlewood method. The main novelty lies in using the structure of bihomogeneous equations to obtain asymptotics in generically fewer variables than would be necessary in using the standard approach for homogeneous varieties. Also, we consider counting functions where not all the variables have to lie in intervals of the same size, which arises as a natural question in the setting of bihomogeneous varieties.
\end{abstract}

\maketitle

\excludecomment{com}

\section{Introduction}

An important issue in the study of diophantine equations is to determine the density of integer points on algebraic varieties. In this setting the circle method is a powerful instrument, with which for example Birch \cite{Bir1961} and Schmidt \cite{Schmidt1985} obtained results in great generality. So far, most literature is concerned with counting integer points in boxes which are dilated by a large real number. In this case all the variables lie in intervals of comparable length. In this paper we study systems of bihomogeneous equations where it is natural to ask for similar asymptotic formulas while allowing different sizes for the variables involved. Furthermore, we use the structure of bihomogeneous equations to obtain results on the number of integer points on these varieties, using in generic cases fewer variables than needed in Birch's work \cite{Bir1961}.\par

First we need to introduce some notation. Let $\na, \nb$ and $R$ be positive integers. We use the vector notation $\bfx = (x_1,\ldots, x_\na)$ and $\bfy = (y_1,\ldots , y_\nb)$. We call a polynomial $F(\bfx;\bfy)\in\Z[\bfx,\bfy]$ a bihomogneneous form of bidegree $(d_1,d_2)$ if 
\begin{equation*}
F(\lam \bfx;\mu \bfy)= \lam^{d_1}\mu^{d_2}F(\bfx;\bfy),
\end{equation*}
for all $\lam,\mu\in \C$ and all vectors $\bfx,\bfy$. In the following we consider a system of bihomogeneous forms $F_i(\bfx,\bfy)\in\Z[\bfx,\bfy]$, for $1\leq i\leq R$. We are interested in the number of solutions to the system of equations
\begin{equation}\label{eqn2.1}
F_i(\bfx;\bfy)=0,
\end{equation}
for $1\leq i\leq R$, where we seek integer solutions in certain boxes. Thus,
let $\calB_1$ and $\calB_2$ be two boxes of side length at most $1$ in $\R^\na$ and $\R^\nb$, and let $P_1$ and $P_2$ be large real numbers. We write $P_1\calB_1$ for the set of $\bfx\in\R^\na$ such that $P_1^{-1}\bfx\in\calB_1$, and $P_2\calB_2$ analogously. Then we define $N(P_1,P_2)$ to be the number of integer solutions to the system of equations (\ref{eqn2.1}) with 
\begin{equation*}
\bfx\in P_1\calB_1 \mbox{ and } \bfy\in P_2\calB_2.
\end{equation*}
Furthermore, we introduce the affine variety $V_1^*$ in $\A_\C^{\na +\nb}$ given
by 
\begin{equation}\label{eqn6.3}
\rank \left( \frac{\partial F_i}{\partial x_j}\right)_{\substack{ 1\leq i\leq
    R\\ 1\leq j\leq \na}} <R.
\end{equation}
Similarly we define $V_2^*$ to be the affine variety in $\A_\C^{\na +\nb}$ given
by
\begin{equation}\label{eqn6.4}
\rank \left(\frac{\partial F_i}{\partial y_j}\right)_{\substack{1\leq i\leq
    R\\ 1\leq j\leq \nb}} <R.
\end{equation}
Our main result is an asymptotic formula for $N(P_1,P_2)$, which we can establish as soon as the codimensions of $V_1^*$ and $V_2^*$ are sufficiently large in terms of the number of equations, the bidegree of the polynomials and the logarithmic ratio between the two parameters $P_1$ and $P_2$. 

\begin{theorem}\label{thm2.1}
Let $P_1$ and $P_2$ be two large real numbers, and define $b=\frac{\log
  P_1}{\log P_2}$. Assume that $b\geq 1$. Furthermore, for all $1\leq i\leq R$, assume that the polynomials $F_i$
have bidegree $(d_1,d_2)$. Let $\na,\nb > R$ and $V_1^*$ and $V_2^*$ be the varieties given by equations (\ref{eqn6.3}) and (\ref{eqn6.4}). Assume that
\begin{equation*}
\na +\nb-\dim V_i^* > 2^{d_1+d_2-2} \max\{ R(R+1) (d_1+d_2-1), R(bd_1+d_2)\},
\end{equation*}
for $i=1,2$. Then we have the asymptotic formula
\begin{equation*}
N(P_1,P_2)= \sig P_1^{\na - R d_1}P_2^{\nb-R d_2} +O(P_1^{\na-Rd_1-\eps}P_2^{\nb-Rd_2}),
\end{equation*}
for some real $\sig$ and $\eps >0$. As usual, $\sig$ is the product of a singular series $\grS$ and a singular integral $J$ which are given in equations (\ref{eqn7.6}) and (\ref{eqn7.7}). Furthermore, the constant $\sig$ is positive if\par
i) the $F_i(\bfx;\bfy)$ have a common non-singular $p$-adic zero for all $p$,\par
ii) and if the $F_i(\bfx;\bfy)$ have a non-singular real zero in the box
$\calB_1\times \calB_2$ and $\dim V(0)= \na +\nb-R$, where $V(0)$ is the
affine variety given by the system of equations (\ref{eqn2.1}).
\end{theorem}

We note that in our result the number of variables $\na$ and $\nb$ depends on the parameter $b$. However, this condition can be omitted if
\begin{equation*}
(R+1)(d_1+d_2-1) \geq (bd_1+d_2).
\end{equation*}
There are few examples in the literature where the number of integer
points on bihomogeneous varieties is studied. Robbiani (\cite{Rob2001}) and Spencer
(\cite{Spe2009}) treat bilinear varieties, and Van Valckenborgh
(\cite{ValA2011}) provides some results on bihomogeneous equations of bidegree
$(2,3)$. However, Van Valckenborgh only considers a diagonal situation,
whereas we are interested in a general set-up.\par

In our work we largely follow Birch's paper \cite{Bir1961}. However, we have
to take care of the different sizes of our boxes and their growth. The main
difference to Birch's work is in the form of Weyl's inequality we use. When
Birch works with forms of total degree $d$ he differentiates them $d-1$ times
via Weyl-differencing to obtain linear exponential sums. We apply that
differencing process separately with respect to the variables $\bfx$ and
$\bfy$, such that we only have to use this process $d_1-1$ times for the
variables $\bfx$ and $d_2-1$ times for the variables $\bfy$. In total we
therefore only need $d_1+d_2-2$ differencing steps. This approach was first mentioned to us by Prof. T. D. Wooley. One condition in Birch's theorem is that the total number of variables $\tilde{n}$ satisfies
\begin{equation*}
\tilde{n}-\dim V^*> R(R+1) (d-1) 2^{d-1},
\end{equation*}
which is essentially determined by the form of Weyl's lemma, which he uses. We obtain a similar condition for $d=d_1+d_2$, however we can replace the factor $2^{d-1}$ by $2^{d-2}$.\par
On the other hand, in our condition the quantities $\dim V_1^*$ and $\dim V_2^*$ appear instead of the dimension of $V^*$, which is the variety given by
\begin{equation*}
\rank \left( \frac{\partial F_i}{\partial z_j}\right) <R,
\end{equation*}
where $z_j$ run through all variables $x_1,\ldots, x_\na$ and $y_1,\ldots,
y_\nb$. We clearly have $V^*\subset V_i^*$ and thus $\dim V^*\leq \dim
V_i^*$, for $i=1,2$. However, we note that the singular locus of a
bihomogeneous variety is rather large, as soon as not both $d_1$ and $d_2$
equal $1$. If we assume for example $d_1>1$, then we see that $V^*$ contains a
linear subspace of dimension $\nb$, when we set $\bfx = 0$. The same holds of
course for $V_1^*$ and $V_2^*$. We assume for the moment that we have $n= \na
=\nb$ and that $d_1$ or $d_2$ is larger than $1$. Then we claim that in a generic situation we have
\begin{equation}\label{eqngen}
n= \dim V^*= \dim V_1^* = \dim V_2^*.
\end{equation}
Since each of the loci has dimension at least $n$, and $V^*\subset V_1^*$, it
suffices by symmetry to show that $\dim V_1^*= n$ in the generic situation.\par
To justify this claim, we note that for fixed bidegree $(d_1,d_2)$ with
$d_1,d_2\geq 1$ there are 
\begin{equation*}
m= {n+d_1-1\choose n-1} {n+d_2-1 \choose n-1}
\end{equation*}
monomials of bidegree $(d_1,d_2)$ in $(\bfx;\bfy)$. We fix an order of them
and associate to each $\bfa \in \A_\Q^m$ a bihomogeneous form
$F_\bfa(\bfx;\bfy)$. We write $\nabla_\bfx F$ for the gradient of a
bihomogeneous form $F(\bfx;\bfy)$ with respect to the variables $\bfx$. For $\bfa \in \P_\Q^{m-1}$ we set 
\begin{equation*}
X_{1,\bfa}= \{ (\bfx;\bfy) \in \P_\Q^{n-1}\times \P_\Q^{n-1}: \nabla_\bfx
F_\bfa (\bfx;\bfy)=0\}.
\end{equation*}
Furthermore, we consider the projective variety
\begin{align*}
\calV = \{ (\bfa;\bfx;\bfy)\in \P_\Q^{m-1}\times \P_\Q^{n-1}\times
\P_\Q^{n-1}: \nabla_\bfx F_\bfa(\bfx;\bfy)=0\},
\end{align*}
and the projection to the first factor $\pi: \calV \rightarrow
\P_\Q^{m-1}$. Define the function 
\begin{equation*}
\lam (\bfa) = \dim (\pi^{-1}(\bfa))= \dim X_{1,\bfa},
\end{equation*}
for $\bfa \in \P_\Q^{m-1}$. Then Corollary 11.13 of \cite{Harris} shows that
$\lam$ is an upper semi-continuous function on $\pi (\calV)$ in the
Zariski-topology of $\pi (\calV)$, which is itself
a closed subset of $\P_\Q^{m-1}$ by Theorem 3.13 of \cite{Harris}. Hence the
set 
\begin{equation*}
Y= \{ \bfa\in \P_\Q^{m-1}: \lam (\bfa) \geq n-1\}
\end{equation*}
is closed in $\pi(\calV)$ and hence in $\P_\Q^{m-1}$. We claim that
$Y\neq \P_\Q^{m-1}$. For this we consider the vector $\bfb \in
\A_\Q^m\setminus \{0\}$
such that 
\begin{equation*}
F_\bfb(\bfx;\bfy)= x_1^{d_1}y_1^{d_2}+\ldots + x_n^{d_1}y_n^{d_2}.
\end{equation*}
Then $X_{1,\bfb}$ is given by $x_iy_i = 0$ for $1\leq i\leq n$ if $d_1\geq 2$,
and empty if $d_1=1$. 
In any case, we have $\dim X_{1,\bfb}\leq n-2$. Therefore the set
\begin{equation*}
\{\bfa\in \P_\Q^{m-1}: \dim X_{1,\bfa}\leq n-2\}
\end{equation*}
is open and non-empty in $\P_\Q^{m-1}$, and so $\dim V_1^*=n$ in the generic case..\par

Another novelty in this work is the way we use of the geometry of numbers in the treatment of our exponential sums. Birch in his paper \cite{Bir1961} uses Lemma 12.6 from \cite{Dav2005}, which is a standard argument at this step. However, this lemma can only be applied if the involved matrices are symmetric, which is not the case in our situation. Our Lemma \ref{lem5.1} provides a form of generalising that lemma from Davenport to general matrices.\par
We note that a system of bihomogeneous polynomials $F_i(\bfx;\bfy)$ defines a
variety in biprojective space $\P^{\na-1}\times \P^{\nb-1}$. Hence, in the
context of the Manin conjectures, it is natural to count rational points on this variety with respect to the anticanonical height function in biprojective space. Our Thoerem \ref{thm2.1} is a first step in this direction and will be used to accomplish this goal in forthcoming work of the author. We note that it will turn out to be important that we can establish asymptotic formulas for $N(P_1,P_2)$ for parameters $P_1$ and $P_2$ which are not necessarily of the same size.\par
In the following $\bfalp$ is some vector $\bfalp = (\alp_1,\ldots, \alp_R)\in\R^R$, and we use the abbreviation $\bfalp \cdot\bfF:= \alp_1 F_1+\ldots +\alp_RF_R$. Furthermore, we frequently use summations over integer vectors $\bfx$ and $\bfy$, such that sums of the type $\sum_{\bfx\in P_1\calB_1}$ are to be understood as sums $\sum_{\bfx\in P_1\calB_1\cap\Z^\na}$. For a real number $x$ we write $\Vert x\Vert = \min_{z\in \Z}|x-z|$ for the distance to the nearest integer. As usual, we write $e(x)$ for $e^{2\pi i x}$.\par
The structure of this paper is as follows. After introducing some notation in section 2, we perform a Weyl-differencing process in section 3. In section 4 we are concerned with the lemma from the geometry of numbers mentioned above. This is used in section 5 to deduce a form of Weyl's inequality. In section 6 we set up the circle method, reduce the problem to a major arc situation and treat the singular series and integral. The proof of Theorem \ref{thm2.1} is finished in the final section.\par
\textbf{Acknowledgements.} During part of the work on this paper the author was
supportet by a DAAD scholarship. Furthermore, the author would like to thank Prof. T. D. Wooley for suggesting this area of research.

\section{Exponential sums}
We start in defining the exponential sum
\begin{equation*}
S(\bfalp)= \sum_{\bfx\in P_1\calB_1}\sum_{\bfy\in P_2\calB_2} e(\bfalp \cdot \bfF (\bfx;\bfy)),
\end{equation*}
for some $\bfalp\in \R^R$. One goal of this section is to perform $(d_1-1)$
times a Weyl-differencing process with respect to the variables $\bfx$ and
$(d_2-1)$ times the same differencing process with respect to $\bfy$. For this
we write each bihomogeneous form $F_i$ as
\begin{equation*}
F_i(\bfx;\bfy)=\sum_{\bfj =1}^\na\sum_{\bfk =1}^\nb F^{(i)}_{j_1,\ldots,
  j_{d_1};k_1,\ldots, k_{d_2}} x_{j_1}\ldots x_{j_{d_1}}y_{k_1}\ldots
y_{k_{d_2}},
\end{equation*}
with the $F^{(i)}_{j_1,\ldots,  j_{d_1};k_1,\ldots, k_{d_2}}$ symmetric in
$(j_1,\ldots, j_{d_1})$ and $(k_1,\ldots, k_{d_2})$. Here the summations are
over $j_1,\ldots,j_{d_1}$ from $1$ to $\na$, and $k_1,\ldots,k_{d_2}$ from $1$ to $\nb$, and we
write $\bfj$ and $\bfk$ for $(j_1,\ldots, j_{d_1})$ and $(k_1,\ldots, k_{d_2})$. Without loss of
generality we can assume the $F^{(i)}_{\bfj;\bfk}$ to be integers (otherwise
multiply with some suitable constant).\par
Let $d_2>1$. We start our differencing process in applying H\"older's inequality to
obtain
\begin{equation}\label{eqneins}
|S(\bfalp)|^{2^{d_2-1}}\ll P_1^{\na(2^{d_2-1}-1)}\sum_{\bfx\in
  P_1\calB_1}|S_\bfx(\bfalp)|^{2^{d_2-1}},
\end{equation}   
with the exponential sum
\begin{equation*}
S_\bfx(\bfalp)=\sum_{\bfy \in P_2\calB_2} e(\bfalp\cdot \bfF (\bfx;\bfy)).
\end{equation*}
Next we use a form of Weyl's inequality as in Lemma 11.1 in
\cite{Schmidt1985} to bound $|S_\bfx(\bfalp)|^{2^{d_2-1}}$. For this we need to introduce some notation. Let $\calU = P_2 \calB_2$, write $\calU^D=\calU-\calU$ for the difference set and define 
\begin{equation*}
\calU (\bfy^{(1)},\ldots,\bfy^{(t)})=\cap_{\eps_1=0}^1\ldots \cap_{\eps_t =
  0}^1 (\calU -\eps_1\bfy^{(1)} -\ldots - \eps_t \bfy^{(t)}).
\end{equation*}
Following the notation of \cite{Schmidt1985}, we define the polynomial $\calF (\bfy)=\bfalp \cdot \bfF (\bfx;\bfy)$. Furthermore we set
\begin{equation*}
\calF_d(\bfy_1,\ldots, \bfy_d)= \sum_{\eps_1=0}^1\ldots \sum_{\eps_d=0}^1 (-1)^{\eps_1+\ldots + \eps_d}\calF (\eps_1\bfy_1+\ldots +\eps_d \bfy_d),
\end{equation*}
and $\calF_0=0$ identically.\par
In our estimate for $|S_\bfx(\bfalp)|^{2^{d_2-1}}$ we want to avoid absolute values in the resulting bound such that we directly consider
equation 11.2 in \cite{Schmidt1985}. This delivers the estimate
\begin{align*}
|S_{\bfx}(\bfalp)|^{2^{d_2-1}}\ll &|\calU^D|^{2^{d_2-1}-d_2}
\sum_{\bfy^{(1)}\in \calU^D} \ldots \\ &\sum_{\bfy^{(d_2-2)}\in\calU^D} \left|
  \sum_{\bfy^{(d_2-1)} \in \calU (\bfy^{(1)},\ldots \bfy^{(d_2-2)})}
  e(\calF_{d_2-1}(\bfy^{(1)},\ldots,\bfy^{(d_2-1)}))\right|^2,
\end{align*}
 We note that all the summation regions for the
$\bfy^{(j)}$ are boxes, since $P_2\calB_2$ is a box and intersections and
differences of boxes are again boxes. As in the proof of Lemma 11.1 in
\cite{Schmidt1985} we consider two elements $\bfz,\bfz'\in \calU
(\bfy^{(1)},\ldots \bfy^{(d_2-2)})$ and note that
\begin{align*}
&\calF_{d_2-1} (\bfy^{(1)},\ldots, \bfz)-\calF_{d_2-1} (\bfy^{(1)},\ldots,
\bfz')\\ =& \calF_{d_2-1}(\bfy^{(1)},\ldots,\bfy^{(d_2-2)},\bfy^{(d_2)})- \calF_{d_2-1}
(\bfy^{(1)},\ldots\bfy^{(d_2-2)},\bfy^{(d_2-1)}+\bfy^{(d_2)})\\
= &\calF_{d_2} (\bfy^{(1)},\ldots,\bfy^{(d_2-1)},\bfy^{(d_2)}) - \calF_{d_2-1}
(\bfy^{(1)},\ldots,\bfy^{(d_2-2)},\bfy^{(d_2-1)}),
\end{align*}
for some $\bfy^{(d_2-1)} \in \calU (\bfy^{(1)},\ldots,\bfy^{(d_2-2)})^D$ and
$\bfy^{(d_2)} \in \calU (\bfy^{(1)},\ldots,\bfy^{(d_2-1)})$. Thus, we obtain
the bound
\begin{align*}
|S_\bfx(\bfalp)|^{2^{d_2-1}} & \ll
P_2^{\nb(2^{d_2-1}-d_2)}\sum_{\bfy^{(1)}\in\calU^D}\ldots
\sum_{\bfy^{(d_2-2)}\in \calU^D} \sum_{\bfy^{(d_2-1)}\in
  \calU(\bfy^{(1)},\ldots,\bfy^{(d_2-2)})^D}\\
& \sum_{\bfy^{(d_2)}\in\calU(\bfy^{(1)},\ldots, \bfy^{(d_2-1)})}
  e(\calF_{d_2}(\bfy^{(1)},\ldots,\bfy^{(d_2)})-\calF_{d_2-1}(\bfy^{(1)},\ldots,\bfy^{(d_2-1)})).
\end{align*}
By Lemma 11.4 of Schmidt's work \cite{Schmidt1985} the polynomial $\calF_{d_2}$ is just the
multilinear form associated to $\calF$. In our case we have
\begin{align*}
&\calF_{d_2}(\bfy^{(1)},\ldots,\bfy^{(d_2)})-\calF_{d_2-1}
(\bfy^{(1)},\ldots,\bfy^{(d_2-1)}) \\ = & \sum_{i=1}^R \alp_i
\sum_{\bfj}\sum_\bfk F_{\bfj,\bfk}^{(i)} x_{j_1}\ldots x_{j_{d_1}} h_\bfk
  (\bfy^{(1)},\ldots, \bfy^{(d_2)} ),
\end{align*}
with 
\begin{equation*}
h_\bfk (\bfy^{(1)},\ldots,\bfy^{(d_2)})= d_2! y_{k_1}^{(1)} \ldots
y_{k_{d_2}}^{(d_2)} + \tilde{h_\bfk}(\bfy^{(1)},\ldots,\bfy^{(d_2-1)}),
\end{equation*}
where $\tilde{h_\bfk}$ are some homogeneous polynomials of degree $d_2$
independent of $\bfy^{(d_2)}$.\par
We come back to estimating $\sum_{\bfx \in P_1\calB_1}
|S_{\bfx}(\bfalp)|^{2^{d_2-1}}$. Set
$\tilde{d}=d_1+d_2-2$. We write  and $\bfytil=(\bfy^{(1)},\ldots,\bfy^{(d_2)})$ and set
\begin{equation*}
S_\bfytil(\bfalp)=\sum_{\bfx\in P_1\calB_1}e\left( \sum_i\alp_i\sum_\bfj\sum_\bfk F_{\bfj,\bfk}^{(i)}x_{j_1}\ldots x_{j_{d_1}}
  h_k(\bfytil)\right).
\end{equation*}
In equation (\ref{eqneins}) we interchange the summation over
$\sum_{\bfx}$ with all the summations $\sum_{\bfy^{(i)}}$ from the bound for $\sum_{\bfx \in P_1\calB_1}
|S_{\bfx}(\bfalp)|^{2^{d_2-1}}$. An application of H\"older's inequality now delivers
\begin{align*}
|S(\bfalp)|^{2^{\tilde{d}}}\ll
P_1^{\na(2^{\tilde{d}}-2^{d_1-1})}P_2^{\nb(2^{\tilde{d}}-d_2)}
\sum_{\bfy^{(1)}}\ldots \sum_{\bfy^{(d_2)}} |S_\bfytil
(\bfalp)|^{2^{d_1-1}}.
\end{align*}

Applying the same differencing process as before to $S_\bfytil(\bfalp)$ leads
us to 
\begin{equation}\label{eqn4.1}
|S(\bfalp)|^{2^{\tilde{d}}}\ll
P_1^{\na(2^{\tilde{d}}-d_1)}P_2^{\nb(2^{\tilde{d}}-d_2)}
\sum_{\bfy^{(1)}}\ldots\sum_{\bfy^{(d_2)}}\sum_{\bfx^{(1)}}\ldots|\sum_{\bfx^{(d_1)}}e(\gam
(\bfxtil;\bfytil))|,
\end{equation}
with 
\begin{equation*}
\gam(\bfxtil;\bfytil)=\sum_i\alp_i\sum_\bfj\sum_\bfk F_{\bfj,\bfk}^{(i)}g_\bfj(\bfxtil)h_\bfk(\bfytil).
\end{equation*}
As before we have
\begin{equation*}
g_\bfj(\bfx^{(1)},\ldots,\bfx^{(d_1)})=d_1! x_{j_1}^{(1)}\ldots
x_{j_{d_1}}^{(d_1)}+\tilde{g}_j(\bfx^{(1)},\ldots,\bfx^{(d_1-1)}),
\end{equation*}
with some homogeneous form $\tilde{g}_\bfj$ of degree $d_1$,and all summations
over $\bfx^{(1)},\ldots,\bfx^{(d_1)}$ run over intervals of length at most
$2P_1$. Note that equation (\ref{eqn4.1}) holds for all integers $d_1\geq 1$
and $d_2\geq 1$. Next we introduce the notation $\bfxhat =
(\bfx^{(1)},\ldots,\bfx^{(d_1-1)})$ and $\bfyhat$ analogously, and turn towards estimating the sum
\begin{equation*}
\sum(\bfxhat,\bfyhat):=\sum_{\bfy^{(d_2)}}\left| \sum_{\bfx^{(d_1)}}e(\gam(\bfxtil;\bfytil))\right|.
\end{equation*}
First we have
\begin{equation*}
\left|\sum_{\bfx^{(d_1)}}e(\gam(\bfxtil,\bfytil))\right|\ll \prod_{l=1}^\na \min
\left(P_1,\Vert\gamtil (\bfxhat,\bfe_l;\bfytil)\Vert^{-1}\right),
\end{equation*}
where $\bfe_l$ is the $l$th unit vector and $\gamtil$ is given by
\begin{equation*}
\gamtil(\bfxtil;\bfytil)=d_1!\sum_i\alp_i\sum_\bfj\sum_\bfk F_{\bfj,\bfk}^{(i)}
x_{j_1}^{(1)}\ldots x_{j_{d_1}}^{(d_1)} h_k(\bfytil).
\end{equation*}
Next we follow Davenport's analysis in \cite{Dav1959}, section 3. For some
real number $z$ we write $\{z\}$ for the fractional part, and use the notation
$\bfr =(r_1,\ldots,r_n)$. For some integers $0\leq r_l <P_1$ let
$\calA(\bfxhat;\bfyhat;\bfr)$ be the set of $\bfy^{(d_2)}$ in the above summation such that
\begin{equation*}
r_lP_1^{-1}\leq \{\gamtil(\bfxhat,\bfe_l;\bfyhat,\bfy^{(d_2)})\}<(r_l+1)P_1^{-1},
\end{equation*}
for $1\leq l\leq \na$. Then we can estimate
\begin{equation*}
\sum_{\bfy^{(d_2)}}\left|\sum_{\bfx^{(d_1)}} e(\gam(\bfxtil;\bfytil))\right|\ll \sum_{\bfr}
A(\bfxhat;\bfyhat;\bfr)\prod_{l=1}^n
\min\left(P_1,\max\left(\frac{P_1}{r_l},\frac{P_1}{P_1-r_l-1}\right)\right),
\end{equation*}
where the summation is over all vectors $\bfr$ with $0\leq r_l<P_1$ for all
$l$, and $A(\bfxhat;\bfyhat;\bfr)$ is the cardinality of the set
$\calA(\bfxhat;\bfyhat;\bfr)$. Our next goal is to find a bound
for $A(\bfxhat;\bfyhat;\bfr)$, which is independent of $\bfr$. For this
consider two vectors $\bfu$ and $\bfv$ counted by that quantity. Then we have
\begin{equation*}
\Vert\gamtil(\bfxhat,\bfe_l;\bfyhat,\bfu)-\gamtil(\bfxhat,\bfe_l;\bfyhat,\bfv)\Vert<P_1^{-1},
\end{equation*}
for $1\leq l\leq \na$. Define the multilinear form 
\begin{equation*}
\Gam (\bfxtil;\bfytil)= d_1!d_2!\sum_i\alp_i\sum_\bfj\sum_\bfk F_{\bfj,\bfk}^{(i)}
x_{j_1}^{(1)}\ldots x_{j_{d_1}}^{(d_1)} y_{k_1}^{(1)}\ldots
y_{k_{d_2}}^{(d_2)},
\end{equation*}
and let $N(\bfxhat;\bfyhat)$ be the number of integer vectors $\bfy\in
(-P_2,P_2)^\nb$ such that 
\begin{equation*}
\Vert\Gam (\bfxhat,\bfe_l;\bfyhat,\bfy) \Vert<P_1^{-1},
\end{equation*}
for all $1\leq l\leq \na$. Observe that 
\begin{equation*}
\gamtil(\bfxhat,\bfe_l;\bfyhat,\bfu)-\gamtil(\bfxhat,\bfe_l;\bfyhat,\bfv)= \Gam (\bfxhat,\bfe_l;\bfyhat,\bfu-\bfv).
\end{equation*}
Thus, we have
\begin{equation*}
A(\bfxhat;\bfyhat;\bfr)\leq N(\bfxhat;\bfyhat),
\end{equation*} 
for all $\bfr$ under consideration. This gives us finally the bound
\begin{equation*}
\sum_{\bfy^{(d_2)}}\left|\sum_{\bfx^{(d_1)}}e(\gam(\bfxtil;\bfytil))\right|\ll
N(\bfxhat;\bfyhat)(P_1\log P_1)^\na.
\end{equation*}
Furthermore, let $M_1(\bfalp; P_1;P_2;P_1^{-1})$ be the number of integer vectors
$\bfxhat\in (-P_1,P_1)^{(d_1-1)\na}$ and $\bfytil \in (-P_2,P_2)^{d_2 \nb}$, such
that 
\begin{equation*}
\Vert \Gam (\bfxhat,\bfe_l;\bfytil)\Vert <P_1^{-1}
\end{equation*}
holds for all $1\leq l\leq \na$. Summing over all $\bfxhat$ and $\bfyhat$ in
equation (\ref{eqn4.1}) gives us the bound
\begin{equation*}
|S(\bfalp)|^{2^{\tilde{d}}}\ll
P_1^{\na(2^\dtil-d_1+1)+\eps}P_2^{\nb(2^\dtil-d_2)} M_1(\bfalp;P_1;P_2;P_1^{-1}).
\end{equation*}
The above discussion delivers now the following lemma.

\begin{lemma}\label{lem4.1}
Let $P$ be a large real number, and $\eps >0$. Then, for some real $\kap >0$, one has either the upper bound
\begin{equation*}
|S(\bfalp)|<P_1^{\na+\eps}P_2^\nb P^{-\kap},
\end{equation*}
or the lower bound 
\begin{equation*}
M_1(\bfalp;P_1;P_2;P_1^{-1})\gg P_1^{\na(d_1-1)}P_2^{\nb d_2}P^{-2^\dtil \kap}.
\end{equation*}
\end{lemma}
Next we want to apply the geometry of numbers to $M_1(\bfalp;P_1;P_2;P_1^{-1})$, similar as done in
Birch's work \cite{Bir1961} in Lemma 2.3 and Lemma 2.4. For this we need a
modified version of a certain lemma from the geometry of numbers which we
give in the following section.

\section{A lemma from the geometry of numbers}

For some integers $\na$ and $\nb$ and real numbers $\lam_{ij}$ for $1\leq
i\leq \na$ and $1\leq j\leq \nb$, we
consider the linear forms
\begin{equation*}
L_i(\bfu)=\sum_{j=1}^\nb \lam_{ij}u_j,
\end{equation*}
and the linear forms corresponding to the transposed matrix of $(\lam_{ij})$
given by
\begin{equation*}
L_j^t(\bfu)=\sum_{i=1}^\na \lam_{ij}u_i.
\end{equation*}
Furthermore, for some real $a>1$ we define $U(Z)$ to be the number of integer
tuples $u_1,\ldots,u_\nb,\ldots,u_{\na+\nb}$, which satisfy
\begin{equation*}
|u_j| <aZ,
\end{equation*}
for $1\leq j\leq \nb$ and
\begin{equation*}
|L_i(u_1,\ldots,u_\nb)-u_{\nb+i}|<a^{-1}Z,
\end{equation*}
for $1\leq i\leq \na$. Let $U^t(Z)$ be defined analogously with $L_i$
replaced by the linear system $L_j^t$. Our goal of this section is to
establish the following lemma using the geometry of numbers.

\begin{lemma}\label{lem5.1}
If $0<Z_1\leq Z_2\leq 1$, then one has the bound
\begin{equation*}
U(Z_2) \ll \max \left( \left(\frac{Z_2}{Z_1}\right)^\nb
  U(Z_1),\frac{Z_2^\nb}{Z_1^\na} a^{\nb-\na}U^t(Z_1)\right).
\end{equation*}
\end{lemma}

In the case of $\na = \nb$ and symmetric coefficients $\lam_{ij}$, i.e. $\lam_{ij}=\lam_{ji}$
for all $i,j$, this is just Lemma 12.6 from \cite{Dav2005}. In our proof we
follow mainly the arguments of Davenport in section 12 of \cite{Dav2005}.

\begin{proof}
We start in defining the lattice $\Gam$ via the matrix
\begin{align*}
\Lam = \left(\begin{array}{cc} a^{-1}I_\nb & 0\\a\lam & aI_\na\end{array}\right),
\end{align*}
where we write $I_n$ for the $n$-dimensional identity matrix and $\lam$ for
the $\na\times \nb$-matrix with entries $\lam_{ij}$. Let $R_1,\ldots, R_{\na+\nb}$ be the
successive minima of $\Lam$. Furthermore consider the adjoint lattice given by
\begin{align*}
M=(\Lam^t)^{-1}=\left(\begin{array}{cc} aI_\nb& -a\lam^t\\0 &
    a^{-1}I_\na \end{array}\right),
\end{align*}
where $\lam^t$ is the transposed matrix of $\lam$. As pointed out by Davenport in
section 12 of \cite{Dav2005}, $M$ has the same
successive minima $S_1,\ldots,S_{\na+\nb}$ as the lattice
\begin{align*}
\Mtil = \left(\begin{array}{cc} a^{-1}I_\na &0\\a\lam^t&aI_\nb\end{array}\right).
\end{align*}
Note that $M$ and $\Lam$ are by
construction adjoint lattices. Set $b=a^{(\nb-\na)/(\na+\nb)}$ and consider the normalised lattices $\Lam^\nor = b \Lam$ and $M^\nor = b^{-1}\Mtil$. Then $\Lam^\nor$ and $M^\nor$ are adjoint lattices of determinant $1$. Let $R_i^\nor$, $1\leq i\leq \na +\nb$ and $S_i^\nor$, $1\leq i\leq \na +\nb$ be the corresponding succissive minima. Then Mahler's lemma (see for
example Lemma 12.5
of \cite{Dav2005}) delivers
\begin{equation*}
R^\nor_k \asymp (S^\nor_{\na+\nb +1-k})^{-1},
\end{equation*}
for all $1\leq k\leq \na +\nb$.\par
We note that $R_i^\nor = b R_i$ and $S_i^\nor = b^{-1}S_i$ for all $i$, and hence we have the relations
\begin{equation*}
R_k \asymp S_{\na+\nb +1-k}^{-1},
\end{equation*}
for all $1\leq k\leq \na +\nb$.\par
Next let $U_0(Z)$ and $U_0^t(Z)$ be the number of lattice points on $\Lam$ and
$\Mtil$, whose euclidean
norm is bounded by $Z$. Then one has
\begin{equation*}
U_0(Z)\leq U(Z) \leq U_0(\sqrt{\na+\nb}Z),
\end{equation*}
and the analogous relation holds for $U^t$ and $U_0^t$. Therefore, we see that
it is enough to establish the bound
\begin{equation*}
U_0(Z_2) \ll_{\na,\nb} \max \left( \left(\frac{Z_2}{Z_1}\right)^\nb
  U_0(Z_1),\frac{Z_2^\nb}{Z_1^\na}a^{\nb-\na}U_0^t(Z_1)\right),
\end{equation*} 
for all $0<Z_1\leq Z_2\leq \sqrt{\na+\nb}$.\par
For this we first assume that $R_1\leq Z_1$ and $S_1\leq Z_1$, and then define the natural numbers $\mu,\nu$ and $\ome$ by
\begin{equation*}
R_\nu\leq Z_1 <R_{\nu+1}, \quad R_\mu\leq Z_2<R_{\mu+1},
\end{equation*}
and
\begin{equation*}
S_\ome \leq Z_1<S_{\ome+1}.
\end{equation*}
Let $U_0^\nor (Z)$ be the number of lattice points on $\Lam^\nor$ with euclidean norm bounded by $Z$. Note that $R_\nu \leq Z_1 < R_{\nu +1}$ is the same as saying that $R_\nu^\nor \leq bZ_1 < R_{\nu +1}^\nor$, and that one has $U_0(Z)= U_0^\nor (bZ)$. Hence Lemma 12.4 of \cite{Dav2005} delivers
\begin{equation*}
U_0(Z_1)=U_0^\nor (bZ_1) \asymp \frac{(bZ_1)^\nu}{R_1^\nor \ldots R_\nu^\nor}= \frac{Z_1^\nu}{R_1\ldots R_\nu}.
\end{equation*}
With the same argument applied to $U_0(Z_2)$ we obtain
\begin{equation*}
\frac{U_0(Z_2)}{U_0(Z_1)}\asymp \frac{Z_2^\mu R_1\ldots R_\nu}{Z_1^\nu
  R_1\ldots R_\mu}.
\end{equation*}
If $\mu\leq \nb$, then we can estimate
\begin{align*}
\frac{U_0(Z_2)}{U_0(Z_1)}\ll \frac{Z_2^\mu}{Z_1^\nu R_{\nu +1}\ldots R_\mu}
\ll \left( \frac{Z_2}{Z_1}\right)^\mu\ll \left(\frac{Z_2}{Z_1}\right)^\nb,
\end{align*}
which is good enough for our lemma. If we have $\mu > \nb$ and $R_{\nb+1}\geq C_1$
for some positive constant $C_1$ to be chosen later, then we have
\begin{equation*}
\frac{Z_2^\mu}{Z_1^\nu R_{\nu +1}\ldots R_\mu}\ll \frac{Z_2^\nb}{Z_1^\nb
  R_{\nb+1}\ldots R_\mu}\ll_{\na,\nb,C_1} \left(\frac{Z_2}{Z_1}\right)^\nb,
\end{equation*}
for $\nu\leq \nb$, and
\begin{equation*}
\frac{Z_2^\mu}{Z_1^\nu R_{\nu +1}\ldots R_\mu} \ll_{C_1} 1 \ll
\left(\frac{Z_2}{Z_1}\right)^\nb,
\end{equation*}
for $\nu >\nb$ using $Z_1\geq R_{\nb+1}\geq C_1$.\par
Next assume $\mu > \nb$ and $R_{\nb +1} <C_1$, and note that we have
$S_\ome\leq Z_1\leq \sqrt{\na +\nb}$. Let $c$ be some positive constant
such that $R_{\nb +1}S_\na> c$. Then we obtain $S_\na >
\tfrac{c}{C_1}$. We set $C_1= c \sqrt{\na +\nb}^{-1}$, which delivers
$S_\na > \sqrt{\na +\nb}$ and thus $\ome< \na$. Now consider
\begin{equation}\label{eqnzwei}
\frac{U_0(Z_2)}{U_0^t(Z_1)}\asymp \frac{Z_2^\mu S_1\ldots S_\ome}{Z_1^\ome
  R_1\ldots R_\mu} \asymp \frac{Z_2^\mu}{Z_1^\ome}(S_1\ldots
S_\ome)(S_{\na +\nb+1-\mu}\ldots S_{\na +\nb}).
\end{equation}  
We use the relation 
\begin{equation*}
S_1\ldots S_{\na+\nb}\asymp b^{\na+\nb}S_1^\nor\ldots S_{\na+\nb}^\nor \asymp b^{\na +\nb}.
\end{equation*}
Hence, if $\ome\leq \na +\nb-\mu$ we can bound the right hand side of equation (\ref{eqnzwei}) by
\begin{equation*}
\ll \frac{Z_2^\mu a^{\nb-\na}}{Z_1^\ome S_{\ome+1}\ldots S_{\na +\nb-\mu}}\ll
\frac{Z_2^\nb a^{\nb-\na}}{Z_1^{\na +\nb-\mu}}\ll \frac{Z_2^\nb}{Z_1^\na}a^{\nb-\na},
\end{equation*}
since $\mu >\nb$ and $Z_1\ll 1$. If $\ome > \na+\nb-\mu$, then we obtain in a similar
way the bound
\begin{align*}
\frac{U_0(Z_2)}{U_0^t(Z_1)}&\ll \frac{Z_2^\mu}{Z_1^\ome} S_{\na +\nb +1-\mu}\ldots
S_\ome a^{\nb-\na}\\ &\ll \frac{Z_2^\nb}{Z_1^\na}Z_1^{\na-\ome}S_{\na+\nb +1-\mu}\ldots
S_\ome a^{\nb-\na}\ll \frac{Z_2^\nb}{Z_1^\na}a^{\nb-\na},
\end{align*}
using $S_\ome\leq Z_1\ll 1$ and $Z_1\ll 1$.\par
If $Z_1<R_1$ or $Z_1<S_1$ the same computations as above show the inequality
which we want to prove, using the observation $U_0(Z_1)=1$ or $U_0^t(Z_1)=1$
in these cases.
\end{proof}

\section{A form of Weyl's inequality}
First we introduce the counting function $M_2(\bfalp;P_1;P_2;P^{-1})$ to be
the number of integer vectors $\bfxtil \in (-P_1,P_1)^{d_1\na}$ and $\bfyhat \in
(-P_2,P_2)^{(d_2-1)\nb}$ such that 
\begin{equation*}
\Vert \Gam(\bfxtil;\bfyhat,\bfe_l)\Vert <P^{-1},
\end{equation*}
for $1\leq l\leq \nb$. Here $P$ is some large real number to be specified later. We need this function for our bounds of
$M_1(\bfalp;P_1;P_2;P^{-1})$, which we introduced in the last section. We
start in writing
\begin{equation*}
M_1(\bfalp;P_1;P_2;P_1^{-1})=\sum_{\bfxhat \in
  (-P_1,P_1)^{(d_1-1)\na}}\sum_{\bfyhat \in (-P_2,P_2)^{(d_2-1)\nb}}
M_{\bfxhat,\bfyhat}(P_2,P_1^{-1}),
\end{equation*}
where $M_{\bfxhat,\bfyhat}(P_2,P_1^{-1})$ is the number of integer vectors
$\bfy^{(d_2)}\in (-P_2,P_2)^{\nb}$ such that
\begin{equation*}
\Vert \Gam(\bfxhat,\bfe_l;\bfyhat,\bfy^{(d_2)})\Vert <P_1^{-1},
\end{equation*}
for $1\leq l\leq \na$. We apply Lemma \ref{lem5.1} to the linear forms
$\Gam(\bfxhat,\bfe_l;\bfyhat,\bfy^{(d_2)})$ in the variables
$\bfy^{(d_2)}$. Let $0< \tet_2\leq 1$ be fixed. We
choose the parameters $Z_1,Z_2$ and $a$ such that
\begin{align*}
P_2&=aZ_2 \quad \quad P_2^{\tet_2}=aZ_1\\
P_1^{-1}&=a^{-1}Z_2.
\end{align*}
This gives $a^{-1}Z_1=P_1^{-1}P_2^{-1+\tet_2}$. Furthermore note that $Z_2\leq
1$ since we have $P_2\leq P_1$.\par
Recall that Lemma \ref{lem5.1} gives a bound of the form
\begin{equation*}
U(Z_2) \ll \max \left( \left(\frac{aZ_2}{aZ_1}\right)^\nb
  U(Z_1),\frac{(aZ_2)^\nb}{(aZ_1)^\na} U^t(Z_1)\right).
\end{equation*}
Hence, we have
\begin{align*}
M_{\bfxhat,\bfyhat}(P_2,P_1^{-1})\ll \max
( &P_2^{\nb(1-\tet_2)} M_{\bfxhat,\bfyhat}(P_2^{\tet_2},P_1^{-1}P_2^{-1+\tet_2}), \\ &P_2^{\nb -\na \tet_2} M^t_{\bfxhat,\bfyhat}(P_2^{\tet_2},P_1^{-1}P_2^{-1+\tet_2})),
\end{align*}
where $M^t_{\bfxhat,\bfyhat}$ counts the solutions of the corresponding
transposed linear system as in section 5. For this we write
\begin{equation*}
\Gam(\bfxhat,\bfe_l;\bfyhat,\bfy^{(d_2)})=\sum_{m=1}^\nb \lam_{lm} y_m^{(d_2)},
\end{equation*}
with 
\begin{equation*}
\lam_{lm}= \Gam(\bfxhat,\bfe_l;\bfyhat,\bfe_m).
\end{equation*}
Still with the notation from section 5 we have
\begin{equation*}
L_m^t(\bfy^{(d_2)})=\sum_{l=1}^\nb\lam_{lm}y^{(d_2)}_l =
\Gam(\bfxhat,\bfy^{(d_2)};\bfyhat,\bfe_m).
\end{equation*}
Therefore, we see that
$M^t_{\bfxhat,\bfyhat}(P_2^{\tet_2},P_1^{-1}P_2^{-1+\tet_2})$ counts the
number of integer vectors $\bfz \in (-P_2^{\tet_2},P_2^{\tet_2})^\na$
with 
\begin{equation*}
\Vert \Gam(\bfxhat,\bfz ;\bfyhat,\bfe_m)\Vert <P_1^{-1}P_2^{-1+\tet_2},
\end{equation*}
for $1\leq m\leq \nb$. Taking the sum over all the contributions of admissible
$\bfxhat$ and $\bfyhat$ we obtain
\begin{equation*}
M_1(\bfalp;P_1;P_2;P_1^{-1})\ll S_1P_2^{\nb(1-\tet_2)}+S_2 P_2^{\nb - \na \tet_2}.
\end{equation*}
Here $S_1$ counts all integer vectors $\bfxhat \in (-P_1,P_1)^{(d_1-1)\na}$ and
$\bfyhat \in (-P_2,P_2)^{(d_2-1)\nb}$ and $\bfz\in
(-P_2^{\tet_2},P_2^{\tet_2})^\nb$ with 
\begin{equation*}
\Vert \Gam(\bfxhat,\bfe_l;\bfyhat,\bfz)\Vert <P_1^{-1}P_2^{-1+\tet_2},
\end{equation*}
for $1\leq l\leq \na$, and $S_2$ is the number of $\bfxhat$ and $\bfyhat$ in the same region and $\bfz\in (-P_2^{\tet_2},P_2^{\tet_2})^\na$ such that
\begin{equation*}
\Vert \Gam(\bfxhat,\bfz;\bfyhat,\bfe_l)\Vert <
P_1^{-1}P_2^{-1+\tet_2},
\end{equation*}
for $1\leq l\leq \nb$.\par
Next we define $\tet_1$ by the relation $P_1^{\tet_1}=P_2^{\tet_2}$ and note
that we have $0 <\tet_1\leq 1$ by the assumption on $P_1$ and $P_2$. For
convenience we write $P_1^{\tet_1}=P^{\tet}$ for some real number $\tet$ and
some $P\geq 2$. Now we iterate the above procedure with repect to all the
vectors from $\bfxhat$ and $\bfyhat$. This delivers the bound
\begin{align*}
M_1(\bfalp;&P_1;P_2;P_1^{-1})\ll P_1^{\na(d_1-1)}P_2^{\nb d_2}P^{-\tet (\na d_1+\nb d_2)}
\\
&\times (P^{\na \tet}M_1(\bfalp;P^\tet;P^\tet;P_1^{-d_1}P_2^{-d_2}P^{\tet (\dtil
  +1)})+P^{\nb \tet}M_2(\bfalp;P^\tet;P^\tet;P_1^{-d_1}P_2^{-d_2}P^{\tet (\dtil +1)})).
\end{align*}
In combination with Lemma \ref{lem4.1} we obtain the following result.

\begin{lemma}\label{lem6.1}
Under the above assumptions one has either the upper bound
\begin{equation*}
|S(\bfalp)|< P_1^{\na+\eps}P_2^\nb P^{-\kap},
\end{equation*}
or the lower bound
\begin{equation*}
M_i(\bfalp;P^\tet;P^\tet;P_1^{-d_1}P_2^{-d_2}P^{\tet (\dtil +1)}) \gg P^{\tet (\na d_1+\nb d_2)-\tet n_i} P^{-2^{\dtil}\kap},
\end{equation*}
for $i=1$ or $i=2$. 
\end{lemma}
Next we proceed similarly as in Birch's work \cite{Bir1961}. We write
\begin{equation*}
\Gam (\bfxtil;\bfytil)=\sum_{i=1}^R\alp_i \Gam_i(\bfxtil,\bfytil),
\end{equation*}
with
\begin{equation*}
\Gam_i(\bfxtil;\bfytil)= d_1!d_2!\sum_\bfj\sum_\bfk F_{\bfj,\bfk}^{(i)}
x_{j_1}^{(1)}\ldots x_{j_{d_1}}^{(d_1)} y_{k_1}^{(1)}\ldots
y_{k_{d_2}}^{(d_2)}.
\end{equation*} 
Suppose that we have some integer vectors $\bfxhat \in
(-P^\tet,P^\tet)^{\na(d_1-1)}$ and $\bfytil \in (-P^\tet,P^\tet)^{\nb d_2}$ counted by $M_1(\bfalp;P^\tet;P^\tet;P_1^{-d_1}P_2^{-d_2}P^{\tet (\dtil
  +1)})$ such
that the matrix
\begin{equation*}
(\Gam_i(\bfxhat,\bfe_l;\bfytil))_{\substack{1\leq i\leq R\\ 1\leq l\leq \na}}
\end{equation*}
has full rank. Without loss of generality we may assume that the leading
$R\times R$ minor has full rank. Our next goal is to show that in this case
the $\alp_i$ are well approximated by rational numbers. For this we write
\begin{equation*}
\Gam(\bfxhat,\bfe_l;\bfytil)=\tilde{a}_l+\deltil_l,
\end{equation*}
for $1\leq l\leq \na$, with some integers $\tilde{a}_l$ and real $\deltil_l$ with
$|\deltil_l|<P_1^{-d_1}P_2^{-d_2}P^{\tet (\dtil+1)}$. Next let $q$ be the
absolute value of the determinant of the matrix
$(\Gam_i(\bfxhat,\bfe_l;\bfytil))_{1\leq i,l\leq R}$, and note that we have
\begin{equation*}
q\ll P^{R\tet (\dtil +1)}.
\end{equation*}
Using the formula for the adjoint matrix of our matrix under consideration we
obtain
\begin{equation*}
\alp_i = q^{-1} (a_i +\del_i),
\end{equation*}
for $1\leq i\leq R$ with some integers $a_i$ and with 
\begin{equation*}
|\del_i| \ll P^{(R-1)\tet (\dtil +1)}\max_l|\deltil_l|.
\end{equation*}
Thus, we obtain the approximation 
\begin{equation*}
|q \alp_i-a_i| \ll P_1^{-d_1}P_2^{-d_2}P^{R\tet (\dtil +1)},
\end{equation*}
for $1\leq i\leq R$.\par
We have now established the following
lemma.

\begin{lemma}\label{lem6.2}
There is some positive constant $C$ such that the following holds. Let $P_2\leq P_1$ and $P$ some real number larger than $2$. Let $0<\tet_2\leq 1$ and write $P_2^{\tet_2}=P^\tet$. Then at least one of the following alternatives
hold.\par
i) One has the upper bound $|S(\bfalp)|<P_1^{\na +\eps}P_2^\nb P^{-\kap}$.\par
ii) There exist integers $1\leq q \leq P^{R(\dtil +1)\tet}$ and $a_1,\ldots,
a_R$ with 
\begin{equation*}
\gcd (q,a_1,\ldots, a_R)=1,
\end{equation*}
and 
\begin{equation*}
2|q\alp_i - a_i| \leq P_1^{-d_1}P_2^{-d_2}P^{R(\dtil +1)\tet},
\end{equation*}
for $1\leq i\leq R$.\par
iii) The number of integer vectors $\bfxhat \in (-P^\tet,P^\tet)^{\na(d_1-1)}$
and $\bfytil \in (-P^\tet,P^\tet)^{\nb d_2}$ such that 
\begin{equation}\label{eqn6.1}
\rank (\Gam_i(\bfxhat,\bfe_l;\bfytil))<R
\end{equation}
is bounded below by 
\begin{equation*}
\geq C (P^\tet)^{\na (d_1-1)+\nb d_2- 2^\dtil\kap/\tet }.
\end{equation*}\par
iv) The number of integer vectors $\bfxtil \in (-P^\tet,P^\tet)^{\na d_1}$
and $\bfyhat \in (-P^\tet,P^\tet)^{\nb (d_2-1)}$ such that 
\begin{equation}\label{eqn6.2}
\rank (\Gam_i(\bfxtil;\bfyhat,\bfe_l))<R
\end{equation}
is bounded below by 
\begin{equation*}
\geq C (P^\tet)^{ \na d_1 +\nb (d_2-1) - 2^\dtil\kap/\tet}.
\end{equation*}
\end{lemma}
We note that the constant $C$ is independent of $\tet_2$.\par
Assume that alternative iii) of the above lemma holds. Let $\calL_1$ be the
affine variety defined by equation (\ref{eqn6.1}) in affine $\na (d_1-1)+\nb d_2$-space. As in Birch's work \cite{Bir1961}, section 3, the condition iii)
implies the lower bound
\begin{equation*}
\dim \calL_1 \geq \na (d_1-1)+\nb d_2  - 2^\dtil \kap/\tet .
\end{equation*}
Recall that the affine variety $V_1^*$ (see equation (\ref{eqn6.3}) in $\A_\C^{\na +\nb}$ is given
by 
\begin{equation*}
\rank \left( \frac{\partial F_i}{\partial x_j}\right)_{\substack{ 1\leq i\leq
    R\\ 1\leq j\leq \na}} <R.
\end{equation*}
Furthermore, let $\calD$ be the linear subspace given by
\begin{equation*}
\bfx^{(1)}=\ldots = \bfx^{(d_1-1)} \mbox{ and } \bfy^{(1)}=\ldots =
\bfy^{(d_2)},
\end{equation*}
in affine $\na (d_1-1)+\nb d_2$-space. Considering these as varieties over the
algebraically closed field $\C$ one has
\begin{equation*}
\dim \calL_1\cap \calD \geq \dim \calL_1 - \nb (d_2-1)-\na (d_1-2).
\end{equation*}
Since $\calL_1\cap \calD$ projects onto $V_1^*$, condition iii) above implies
\begin{equation*}
\dim V_1^* \geq \na +\nb - 2^\dtil \kap /\tet .
\end{equation*}
Similarly, we note that condition iv) of Lemma \ref{lem6.2} implies
\begin{equation*}
\dim V_2^* \geq \na +\nb - 2^\dtil \kap /\tet .
\end{equation*}
Define $K$ by
\begin{equation*}
2^{\dtil} K= \min \{ \na +\nb-\dim V_1^*, \na +\nb-\dim V_2^*\}.
\end{equation*}
Furthermore we set $P=P_1^{d_1}P_2^{d_2}$ for the rest of this paper. Note
that this gives the relations
\begin{equation*}
\tet = (b d_1+d_2)^{-1}\tet_2,
\end{equation*}
and
\begin{equation*}
\tet_1=b^{-1}\tet_2.
\end{equation*}
Next we define $\grM (\tet)$ to be the set of $\bfalp \in [0,1]^R$ such that
$\bfalp$ satisfies condition ii) of Lemma \ref{lem6.2}. With this notation we
can state our final lemma of this section, which is a direct consequence of
Lemma \ref{lem6.2}.

\begin{lemma}\label{lem6.3}
Let $0 <\tet \leq (bd_1+d_2)^{-1}$ and assume $\eps >0$. Then one has
for some
real vector $\bfalp \in \R^R$ either $\bfalp \in \grM(\tet)$ modulo $1$ or the
upper bound
\begin{equation*}
|S(\bfalp)| \ll P_1^{n_1}P_2^{n_2} P^{-K\tet +\eps}.
\end{equation*}
\end{lemma}

\section{Circle method}

In this section we set up the circle method to get an asymptotic formula for
$N(P_1,P_2)$ mainly following Birch's work \cite{Bir1961}. We note that by orthogonality we have
\begin{equation}\label{eqn7.1}
N(P_1,P_2)= \int_{[0,1]^R} S(\bfalp)\d\bfalp.
\end{equation}
In the following we assume that we have
\begin{equation}\label{eqn7.2}
K> \max\{R(R+1) (\dtil +1),R(bd_1+d_2)\}.
\end{equation}
Next we choose positive and real $\del$ and $\vartet_0$ in such a way that the
following conditions are satisfied
\begin{equation}\label{eqn7.3}
K- R(R+1)(\dtil+1) > 2 \del \vartet_0^{-1},
\end{equation}
\begin{equation}\label{eqn7.4}
K> (2\del +R) (bd_1+d_2),
\end{equation}
and
\begin{equation}\label{eqn7.5}
1 > (bd_1+d_2) R(\dtil +1) \vartet_0 (2R+3) +\del (bd_1+d_2).
\end{equation}
Note that the parameters $\del$ and $\vartet_0$ may depend on $b$. Now we use
the results of the last section to show that the contribution of those
$\bfalp$ which are not in $\grM (\vartet_0)$ is neglegible in equation
(\ref{eqn7.1}). This is done in the following lemma.

\begin{lemma}\label{lem7.1}
One has
\begin{equation*}
\int_{\bfalp\notin \grM(\vartet_0)} |S(\bfalp)|\d\bfalp =
O(P_1^{n_1}P_2^{n_2} P^{-R-\del}).
\end{equation*}
\end{lemma}

\begin{proof}
We choose a sequence of $\vartet_i$ with 
\begin{equation*}
\vartet_T > \vartet_{T-1} > \ldots > \vartet_1 >\vartet_0 >0,
\end{equation*}
and 
\begin{equation*}
\vartet_T \leq (bd_1+d_2)^{-1} \quad \mbox{ and } \quad \vartet_T K > 2\del
+R.
\end{equation*}
Note that this is possible by equation (\ref{eqn7.4}). Furthermore we choose our $\vartet_i$ in such a way that they
satisfy
\begin{equation*}
\frac{1}{2}\del > R(R+1) (\dtil +1) (\vartet_{t+1} - \vartet_t),
\end{equation*}
for $0\leq t < T$. We certainly can achieve this with $T\ll P^{\del/2}$.\par
Now we consider the contribution of those $\bfalp$, which do not belong to
$\grM(\vartet_T)$. By Lemma \ref{lem6.3} we have
\begin{align*}
\int_{\bfalp\notin \grM(\vartet_T)} |S(\bfalp)|\d\bfalp &\ll P_1^{n_1}P_2^{n_2}
P^{-K\vartet_T +\eps} \\
& \ll P_1^{n_1}P_2^{n_2} P^{-R-\del}.
\end{align*}
For some $\tet >0$ we can estimate the measure of $\grM(\tet)$ by
\begin{align*}
\meas (\grM(\tet)) &\ll \sum_{q\leq P^{R(\dtil +1) \tet}}\sum_\bfa q^{-R}
P_1^{-d_1R} P_2^{-d_2 R} P^{R^2 (\dtil +1)\tet}\\ &\ll  P^{-R+R(R+1) (\dtil
  +1) \tet}.
\end{align*}
This estimate together with Lemma \ref{lem6.3} delivers the bound
\begin{equation*}
  \int_{\bfalp \in \grM(\vartet_{t+1})\setminus \grM(\vartet_{t})}
  |S(\bfalp)|\d\bfalp \ll P_1^{n_1} P_2^{n_2} P^{-K \vartet_t+\eps - R +R
    (R+1)(\dtil +1) \vartet_{t+1}}.
\end{equation*}
Since we have the inequality
\begin{equation*}
-K\vartet_t +R(R+1) (\dtil +1) \vartet_{t+1} \leq \frac{1}{2}\del +\vartet_t
(-K + R(R+1) (\dtil +1)) \leq \frac{1}{2} \del - 2 \del,
\end{equation*}
we finally obtain the estimate
\begin{equation*}
\int_{\bfalp \in \grM(\vartet_{t+1})\setminus \grM(\vartet_{t})}
  |S(\bfalp)|\d\bfalp \ll P_1^{n_1}P_2^{n_2} P^{-R-3 \del/2},
\end{equation*}
for $0\leq t < T$, which is enough to prove the lemma.
\end{proof} 

Next we turn towards the contribution of the major arcs. In order to obtain
nicer formulas, we first
define some modified major arcs. For some $q$ and $0\leq a_i <q$ let $\grM_{\bfa,q}'(\tet)$ be the set of
$\bfalp \in [0,1]^R$ such that 
\begin{equation*}
|q\alp_i - a_i| \leq q P^{-1+ R(\dtil +1) \tet},
\end{equation*}
for $1\leq i\leq R$. In the same way as before we set 
\begin{equation*}
\grM'(\tet) = \bigcup_{1\leq q \leq P^{R(\dtil +1) \tet}} \bigcup_\bfa
\grM'_{\bfa,q} (\tet),
\end{equation*}
where the union for the $\bfa$ is over all $0\leq a_i < q$ with $\gcd
(q,a_1,\ldots, a_R)=1$. We note that the $\grM'_{\bfa,q}(\tet)$ are disjoint if
$\tet$ is sufficiently small. If we have in the above union some 
\begin{equation*}
\bfalp \in \grM'_{\bfa,q} (\tet) \cap \grM'_{\tilde{\bfa},\tilde{q}}(\tet),
\end{equation*}
for distinct $\bfa,q$ and $\tilde{\bfa},\tilde{q}$, then there is some $1\leq i\leq R$ such that 
\begin{equation*}
\frac{1}{ q\tilde{q} } \leq \left|\frac{a_i}{q}-\frac{\tilde{a}_i}{\tilde{q}}\right|
\leq 2 P^{-1+ R(\dtil +1)\tet}.
\end{equation*}
This is impossible for large $P$ and $\tet < 1/ (3R (\dtil +1))$. By equation
(\ref{eqn7.5}) we see that our major arcs $\grM'(\vartet_0)$ are
disjoint. Thus, we have the following lemma, which is a direct consequence of
Lemma \ref{lem7.1} and equation (\ref{eqn7.1}).

\begin{lemma}\label{lem7.2}
One has
\begin{equation*}
N(P_1,P_2)= \sum_{1\leq q \leq P^{R (\dtil +1) \vartet_0}} \sum_{\bfa}
\int_{\grM'_{\bfa,q} (\vartet_0)} S(\bfalp) \d\bfalp +O(
P_1^{n_1}P_2^{n_2} P^{-R-\del}),
\end{equation*}
where the second sum is over all $0\leq a_i < q$ for $1\leq i\leq R$, such that 
\begin{equation*}
\gcd(q,a_1,\ldots, a_R)=1.
\end{equation*}
\end{lemma}

Our next goal is to obtain an approximation for $S(\bfalp)$ on the major
arcs. For convenience we write in the following $\eta =
R(\dtil+1)\vartet_0$. Furthermore, for some $\bfalp \in
\grM'_{\bfa,q}(\vartet_0)$ we write $\bfalp = \bfa/q +\bfbet$ with
\begin{equation*}
|\bet_i| \leq P^{-1+\eta},
\end{equation*}
for $1\leq i\leq R$. We introduce the notation 
\begin{equation*}
S_{\bfa,q} = \sum_{\bfx,\bfy} e\left(\sum_{i=1}^R a_i F_i(\bfx,\bfy)/q\right),
\end{equation*}
where $\bfx$ and $\bfy$ run through a complete set of residues modulo $q$.
Let
\begin{equation*}
I(\bfu)= \int_{\calB_1\times \calB_2} e\left(\sum_{i=1}^R u_i F_i(\bfv;\bfw)\right)
\d\bfv\d\bfw,
\end{equation*}
for some real vector $\bfu = (u_1,\ldots , u_R)$. Now we have introduced all
the notation we need to state our next lemma.

\begin{lemma}\label{lem7.3}
Let $\bfalp \in \grM'_{\bfa,q}(\vartet_0)$ and $q\leq P^\eta$. Then one has
\begin{equation*}
S(\bfalp)= P_1^{n_1}P_2^{n_2} q^{-\na -\nb} S_{\bfa,q} I(P\bfbet) + O(P_1^{n_1}P_2^{n_2}
P^{2\eta}P_2^{-1}).
\end{equation*}
\end{lemma}

\begin{proof}
In the sum $S(\bfalp)$ we write $\bfx = \bfz^{(1)}+ q \bfx'$ and $\bfy
= \bfz^{(2)}+ q \bfy'$, with $0\leq z_i^{(1)} < q$ and $0\leq z_i^{(2)} <
q$ for all $1\leq i\leq n$. Then we obtain
\begin{align*}
S(\bfalp)&= \sum_{\bfx\in P_1\calB_1}\sum_{\bfy\in P_2\calB_2} e\left(\sum_{i=1}^R
\alp_i F_i(\bfx;\bfy)\right)\\
&= \sum_{\bfz^{(1)}}\sum_{\bfz^{(2)}} e\left(\sum_{i=1}^R a_i
F_i(\bfz^{(1)};\bfz^{(2)})/q\right) S_3 (\bfz^{(1)},\bfz^{(2)}),
\end{align*}
with the sum
\begin{equation*}
S_3 (\bfz^{(1)},\bfz^{(2)})= \sum_{\bfx'}\sum_{\bfy'} e\left(\sum_{i=1}^R\bet_i
F_i(q\bfx' + \bfz^{(1)};q \bfy' +\bfz^{(2)})\right),
\end{equation*}
where the integer vectors $\bfx'$ run through a range such that $q\bfx'
+\bfz^{(1)} \in P_1\calB_1$ and for $\bfy'$ analogously.\par
Consider some vectors $\bfx',\bfx''$ and $\bfy',\bfy''$ with
\begin{equation*}
\max_{1\leq i\leq n_1}|x_i' - x_i''| \leq 2,
\end{equation*}
and
\begin{equation*}
  \max_{1\leq i\leq n_2}|y_i'-y_i''|\leq 2.
\end{equation*}
In this case one has
\begin{align*}
|F_i(q\bfx' +\bfz^{(1)};q\bfy' +\bfz^{(2)})-F_i(q \bfx''
+\bfz^{(1)};q\bfy'' + \bfz^{(2)})|&\ll
qP_1^{d_1-1}P_2^{d_2}+qP_1^{d_1}P_2^{d_2-1} \\ &\ll qP_1^{d_1}P_2^{d_2-1}.
\end{align*}
We replace the sum in $S_3$ with an integral and
obtain
\begin{align*}
S_3&=  \int_{q\tilde{\bfv}\in
  P_1\calB_1}\int_{q\tilde{\bfw}\in P_2\calB_2} e\left(\sum_{i=1}^R\bet_i
F_i(q\tilde{\bfv}; q \tilde{\bfw})\right)\d\tilde{\bfv}\d\tilde{\bfw} \\ &+ O\left( \sum_{i=1}^R
|\bet_i| q P_1^{d_1} P_2^{d_2-1}
\left(\frac{P_1}{q}\right)^{\na}\left(\frac{P_2}{q}\right)^\nb +
\left(\frac{P_1}{q}\right)^\na \left(\frac{P_2}{q}\right)^{\nb-1}\right).
\end{align*}
A variable substitution $\bfv= qP_1^{-1}\bfvtil$ and $\bfw = q P_2^{-1}\bfwtil$ in the integral leads to
\begin{align*}
S_3&= P_1^\na P_2^\nb q^{-(\na +\nb) } \int_{\bfv \in \calB_1}\int_{\bfw \in \calB_2}
e\left(\sum_{i=1}^R P_1^{d_1}P_2^{d_2} \bet_i F_i(\bfv;\bfw)\right)\d\bfv\d\bfw\\
&+O(q^{-\na -\nb+1}P^{\eta} P_2^{-1}P_1^{\na}P_2^\nb+q^{-\na -\nb+1}P_1^\na P_2^{\nb-1})\\ 
&= P_1^\na P_2^\nb q^{-\na -\nb}  I(P\bfbet)
+O(P_1^\na P_2^\nb P^\eta P_2^{-1} q^{-\na -\nb+1}).
\end{align*}
Summing over $\bfz^{(1)}$ and $\bfz^{(2)}$ we finally obtain the approximation
\begin{equation*}
S(\bfalp)= P_1^\na P_2^\nb q^{-\na -\nb} S_{\bfa,q}
I( P\bfbet) +O(P_1^\na P_2^\nb P^{2\eta}P_2^{-1}),
\end{equation*}
as desired. 
\end{proof}

Now we use the approximation of Lemma \ref{lem7.3} to evaluate the sum over the major arcs from Lemma \ref{lem7.2}. This leads to
\begin{align*}
N(P_1,P_2)= &P_1^\na P_2^\nb \sum_{1\leq q\leq P^\eta}q^{-\na -\nb} \sum_{\bfa}
S_{\bfa,q} \int_{|\bfbet|\leq P^{-1+\eta}}I(P\bfbet) \d\bfbet \\ & +O
(P_1^\na P_2^\nb P^{2\eta}P_2^{-1} \meas (\grM'(\vartet_0))).
\end{align*}
The measure of these major arcs is bounded by
\begin{equation*}
\meas (\grM'(\vartet_0))\ll \sum_{q\leq P^\eta} q^R P^{-R+\eta R}\ll P^{-R+\eta (2R+1)}.
\end{equation*}
We define the sum
\begin{equation*}
\grS (P^\eta) = \sum_{1\leq q\leq P^\eta}q^{-\na -\nb}\sum_{\bfa} S_{\bfa,q},
\end{equation*}
where the second sum is as before over all tuples $0\leq a_i <q$ with $\gcd (q,a_1,\ldots, a_R)=1$, and we define the integral
\begin{equation*}
J(P^\eta)= \int_{|\bfbet|\leq P^\eta}I(\bfbet) \d\bfbet.
\end{equation*}
With this notation we obtain
\begin{align*}
N(P_1,P_2)&= P_1^\na P_2^\nb P^{-R} \grS(P^\eta) \int_{|\bfbet|\leq P^\eta}
I(\bfbet)\d\bfbet +O(P_1^\na P_2^\nb P^{-R}P_2^{-1}P^{\eta (2R+3)})\\ &=
P_1^\na P_2^\nb P^{-R} \grS(P^\eta) J(P^\eta) +O(P_1^\na P_2^\nb P^{-R+\eta (2R+3)-1/(bd_1+d_2)}).
\end{align*}
The error term is bounded by $O(P_1^nP_2^nP^{-R-\del})$ if we have
\begin{equation*}
\tfrac{1}{b d_1+d_2}> \eta (2R+3) +\del,
\end{equation*}
which is just equation (\ref{eqn7.5}). Thus, we have obtained the following asymptotic for $N(P_1,P_2)$.
\begin{lemma}\label{lem7.4} 
Assume that equation (\ref{eqn7.2}) holds and let $\del$ and $\vartet_0$ be chosen as at the beginning of this section. Then one has
\begin{equation*}
N(P_1,P_2)= P_1^\na P_2^\nb P^{-R} \grS(P^\eta) J(P^\eta) +O(P_1^\na P_2^\nb P^{-R-\del}).
\end{equation*}
\end{lemma}

Next we consider the terms $\grS(P^\eta)$ and $J(P^\eta)$ separately. First we define the singular series,
\begin{equation}\label{eqn7.6}
\grS = \sum_{q=1}^\infty \sum_{\bfa} q^{-(\na +\nb)} S_{\bfa,q},
\end{equation}
if this series exists. The following lemma shows that this is the case, and that $\grS$ is absolutely convergent.

\begin{lemma}\label{lem7.5}
The series $\grS$ is absolutely convergent and one has
\begin{equation*}
|\grS(Q)-\grS|\ll Q^{-\del/\eta},
\end{equation*}
for any large real number $Q$.
\end{lemma}

\begin{proof}
First we need an estimate for the sums $S_{\bfa,q}$. For this we note that we have
\begin{equation*}
S_{\bfa,q}= S(\bfalp),
\end{equation*}
if we set $\calB_1 = [0,1)^\na$, $\calB_2=[0,1)^\nb$ and $P_1=P_2=q$ and $\bfalp = \bfa/q$. We define $\tet$ by
\begin{equation*}
(d_1+d_2) R(\dtil +1) \tet= 1-\eps,
\end{equation*}
for some $\eps >0$. Then we claim that $\bfa/q$ cannot lie inside the major arcs $\grM(\tet)$, if we assume $\gcd (q,a_1,\ldots, a_R)=1$. Otherwise we would have some integers $q'$ and $\bfa'$ with
\begin{equation*}
1\leq q' \leq q^{(d_1+d_2)R(\dtil +1) \tet},
\end{equation*}
and
\begin{equation*}
2 |q'a_i-a_i'q|\leq q q^{-d_1}q^{-d_2}q^{(d_1+d_2)R(\dtil+1) \tet},
\end{equation*}
for all $1\leq i\leq R$, which is impossible. Therefore Lemma \ref{lem6.3} delivers
\begin{align*}
|S_{\bfa,q}| &\ll q^{\na +\nb } q^{-K (d_1+d_2)[(d_1+d_2)R(\dtil +1)]^{-1}+\eps} \\
& \ll q^{\na +\nb  - K/(R (\dtil+1))+\eps}.
\end{align*}
With equation (\ref{eqn7.3}) this leads to the bound
\begin{equation*}
|S_{\bfa,q}| \ll q^{\na +\nb-R-1-\del/\eta}.
\end{equation*}
Now we can estimate the desired series
\begin{equation*}
\sum_{q >Q}\sum_{\bfa} q^{-\na -\nb}|S_{\bfa,q}| \ll \sum_{q>Q} q^{-1-\del/\eta}\ll Q^{-\del/\eta},
\end{equation*}
which proves both claims of the lemma.
\end{proof}

Similarly as for the singular series, we define the singular integral
\begin{equation}\label{eqn7.7}
J= \int_{\bfbet \in \R^R} I(\bfbet)\d\bfbet,
\end{equation}
if this exists.

\begin{lemma}\label{lem7.6}
The singular integral $J$ is absolutely convergent and we have
\begin{equation*}
|J- J(\Phi)| \ll \Phi^{-1},
\end{equation*}
for any large positive real number $\Phi$.
\end{lemma}

\begin{proof}
For convenience of notation we set $B= \max_i |\bet_i|$ for some real vector $\bfbet = (\bet_1,\ldots,\bet_R)$, and assume $B\geq 2$. Set $\tet=\vartet_0$ as we have chosen it at the beginning of this section and define $P$ by 
\begin{equation*}
2B= P^{R(\dtil +1)\tet}.
\end{equation*}
Then we have $P^{-1}\bfbet \in \grM_{0,1}(\tet)$, since
\begin{equation*}
2|P^{-1}\bet_i|\leq P^{-1}P^{R(\dtil +1)\tet},
\end{equation*}
for all $1\leq i\leq R$. Then Lemma \ref{lem7.3} delivers 
\begin{equation}\label{eqnlem7.6}
S(P^{-1}\bfbet)= P_1^\na P_2^\nb I(\bfbet) +O(P_1^\na P_2^\nb P^{2R(\dtil +1)\tet} P_2^{-1}).
\end{equation}
Furthermore $P^{-1}\bfbet$ lies by construction on the boundary of $\grM (\tet)$, which are disjoint by Lemma 4.1 of Birch's paper \cite{Bir1961}. Thus, our Lemma \ref{lem6.3} gives the bound
\begin{equation*}
|S(P^{-1}\bfbet)| \ll P_1^\na P_2^\nb P^{-K\tet +\eps}.
\end{equation*}
Together with equation (\ref{eqnlem7.6}) this implies
\begin{equation*}
|I(\bfbet)| \ll P^{-K\vartet_0+\eps}+P^{2R(\dtil+1)\tet - 1/(bd_1+d_2)}.
\end{equation*}
From equation (\ref{eqn7.5}) we see that
\begin{equation*}
\tfrac{1}{bd_1+d_2} - 2R (\dtil +1)\vartet_0 > 2R (R+1)(\dtil +1)\vartet_0 +\del,
\end{equation*}
which implies
\begin{equation*}
P^{2R(\dtil+1)\tet - 1/(bd_1+d_2)}\ll B^{-2R}.
\end{equation*}
In the same way we see that equation (\ref{eqn7.3}) gives 
\begin{equation*}
P^{-K\vartet_0+\eps}\ll B^{-R-1},
\end{equation*}
such that we have
\begin{equation*}
|I(\bfbet)| \ll (\max_i|\bet_i|)^{-R-1}.
\end{equation*}
Now we can use this bound to estimate the integral
\begin{equation*}
\int_{\Phi_1\leq B\leq \Phi_2} |I(\bfbet)|\d\bfbet \ll \int_{\Phi_1\leq B\leq \Phi_2} B^{R-1}B^{-R-1}\d B \ll \Phi_1^{-1}.
\end{equation*}
This shows that $J$ is absolutely convergent and also that the second assertion of the lemma holds.
\end{proof}

\section{Conclusions}
Before we finish our proof of Theorem \ref{thm2.1}, we give an alternative representation of the singular integral, following Schmidt's work \cite{Schmidt1981}. For this we define the function
\begin{align*}
\psi(z)=\left\{ \begin{array}{cc} 1-|z| & \mbox{ for } |z|\leq 1,\\ 0 & \mbox{ for } |z| >1,\end{array}\right.
\end{align*}
and for $T>0$ we set $\psi_T(z)=T\psi (T z)$. Furthermore, for some vector $\bfz = (z_1,\ldots,z_R)$ we define 
\begin{equation*}
\psi_T(\bfz)= \psi_T(z_1)\cdot\ldots\cdot \psi_T(z_R).
\end{equation*}
With this notation we define
\begin{equation*}
\tilde{J}_T = \int_{\calB_1\times \calB_2} \psi_T(\bfF(\bfxi^{(1)};\bfxi^{(2)}))\d\bfxi^{(1)}\d\bfxi^{(2)},
\end{equation*}
and
\begin{equation*}
\tilde{J}=\lim_{T\rightarrow \infty}\tilde{J}_T,
\end{equation*}
if the limit exists.

\begin{proof}[Proof of Theorem \ref{thm2.1}]
Note that the assumptions of Theorem \ref{thm2.1} imply that equation (\ref{eqn7.2}) holds. Hence, by Lemma \ref{lem7.4} we have
\begin{equation*}
N(P_1,P_2)= P_1^\na P_2^\nb P^{-R} \grS(P^\eta) J(P^\eta) +O(P_1^\na P_2^\nb P^{-R-\del}).
\end{equation*}
Together with Lemma \ref{lem7.5} and Lemma \ref{lem7.6} this gives
\begin{equation*}
N(P_1,P_2)= P_1^\na P_2^\nb P^{-R} \grS J +O(P_1^\na P_2^\nb P^{-R-\del}),
\end{equation*}
which already proves the first part of the theorem.\par
As usual, the singular series $\grS$ factorizes as $\grS = \prod_p \grS_p$, where the product is over all primes $p$, and 
\begin{equation*}
\grS_p = \sum_{l=1}^\infty \sum_{\bfa} p^{-(\na +\nb)l}S_{\bfa,p^l},
\end{equation*}
where the sum over $\bfa$ is over all $0\leq a_i < p^l$ with $\gcd (a_1,\ldots, a_R,p)=1$. We know in a relatively general context that $\grS >0$ if the $F_i(\bfx;\bfy)$ have a common non-singular $p$-adic zero for all $p$. This can for example be found in Birch's work \cite{Bir1961}, and applies to our case, since $\grS$ is absolutely convergent by Lemma \ref{lem7.5}.\par
Our singular integral can be treated in the very same way as in Schmidt's work
\cite{Schmidt1981}. First of all we know that $\tilde{J} > 0$, if $\dim V(0)=
\na +\nb-R$ and if the $F_i(\bfx;\bfy)$ have a non-singular real zero in $\calB_1\times \calB_2$. This is just Lemma 2 from Schmidt's paper \cite{Schmidt1981}. Furthermore, we have shown in the proof of Lemma \ref{lem7.6} that we have
\begin{equation*}
|I(\bfbet)|\ll \min (1,\max_i|\bet_i|^{-R-1}),
\end{equation*}
which enables us to apply section 11 of \cite{Schmidt1981}. This implies that the limit 
\begin{equation*}
\tilde{J}=\lim_{T\rightarrow\infty} \tilde{J}_T
\end{equation*}
exists and equals $\tilde{J}=J$. This proves our main theorem.
\end{proof}

\bibliographystyle{amsbracket}

\begin{thebibliography}{18}


\bibitem{Bir1961}
B. J. Birch, \emph{Forms in many variables}, Proc. Roy. Soc. Ser. A \textbf{265} (1961), 245--263.

\bibitem{Dav1959}
H. Davenport, \emph{Cubic Forms in Thirty-Two Variables},
Phil. Trans. R. Soc. Lond. A \textbf{251} (1959), 193--232.

\bibitem{Dav2005}
H. Davenport, \emph{Analytic methods for Diophantine equations and Diophantine
  inequalities}, Cambridge Mathematical Library. Cambridge University Press,
Cambridge, second edition, 2005. With a foreword by R. C. Vaughan,
D. R. Heath-Brown and D. E. Freeman, Edited and prepared for publication by
T. D. Browning.

\bibitem{Harris}
J. Harris, \emph{Algebraic Geometry, A First Course}, Springer (1993).

\bibitem{Rob2001}
M. Robbiani, \emph{On the number of rational points of bounded height on
  smooth bilinear hypersurfaces in biprojective space}, J. London
Math. Soc. \textbf{63} (2001), 33--51.

\bibitem{Schmidt1981}
W. M. Schmidt, \emph{Simultaneous rational zeros of quadratic forms}, Seminar Delange-Pisot-Poitou 1981. Progress in Math. Vol 22 (1982), 281--307.

\bibitem{Schmidt1985}
W. M. Schmidt, \emph{The density of integer points on homogeneous varieties},
Acta Math. \textbf{154} (1985), no. 3-4 243--296.

\bibitem{Spe2009}
C. V. Spencer, \emph{The Manin conjecture for $x_0y_0+\ldots+x_sy_s=0$},
J. Number Theory \textbf{129} (2009), no. 6, 1505--1521.

\bibitem{ValA2011}
K. van Valckenborgh, \emph{Squareful numbers in hyperplanes}, arXiv 2011, 1001.3296v3. 

\bibitem{Vau1997}
R. C. Vaughan, \emph{The Hardy-Littlewood method}, volume 125 of Cambridge
Tracts in Mathematics, Cambridge University Press, Cambridge, second edition, 1997.

\end{thebibliography}
\providecommand{\bysame}{\leavevmode\hbox to3em{\hrulefill}\thinspace}

\end{document}